\documentclass[a4paper,12pt]{article}
\usepackage{cmap}                        
\usepackage[cp1251]{inputenc}            
\usepackage[english]{babel}
\usepackage[left=2cm,right=2cm,top=2cm,bottom=2cm]{geometry} 
\usepackage[ruled,vlined]{algorithm2e}
\usepackage{amssymb}
\usepackage{amsmath, amsthm}
\theoremstyle{plain}
\newtheorem{thm}{Theorem}
\newtheorem{lem}{Lemma}
\newtheorem{prop}{Proposition}
\newtheorem{cor}{Corollary}

\theoremstyle{definition}

\newtheorem{defn}{Definition}
\newtheorem{pr}{Problem}

\begin{document}

\begin{center}\Large
\textbf{Formations of Finite Groups in Polynomial Time:
$\mathfrak{F}$-residuals and $\mathfrak{F}$-subnormality}\normalsize

\smallskip
Viachaslau I. Murashka

 \{mvimath@yandex.ru\}

Department of  Mathematics and Technologies of Programming,

 Francisk Skorina Gomel State University, Gomel, Belarus\end{center}

 \begin{abstract}
  For a wide family of formations $\mathfrak{F}$ it is proved that the $ \mathfrak{F}$-residual of a permutation finite group can be computed in a polynomial time. Moreover, if in the previous case $\mathfrak{F}$ is hereditary, then an $\mathfrak{F}$-subnormality of a subgroup can be checked in a polynomial time.
 \end{abstract}



\textbf{AMS}(2010). 20D10, 20B40.

\section*{Introduction and the Main Results}

All groups considered here are finite. 
One of the central directions in the modern algebra is the study of different classes of algebraic systems (groups, semigroups, rings, Lie algebras and other). The main problems of it are to construct classes of algebraic systems, to study the structure of a given system in such class and to find wether a given system belongs to a given class or not.

This direction is well illustrated in the rather developed theory of classes of finite groups (formations, Schunk and Fitting classes).   The main results in this direction are presented in monographs of Shemetkov \cite{Shemetkov1978}, Doerk and Hawkes \cite{Doerk1992}, Ballester-Bolinches and Ezquerro \cite{BallesterBollinches2006}, Wenbin Guo \cite{Guo2015}  and others. According to zbMATH Open \cite{zb2021} (formerly known as Zentralblatt MATH) there are more than 5000 papers in this direction (20D10).


The computational theory of classes of finite groups is not as developed as its theoretical part. The main results of this theory are presented in the papers \cite{EICK2002} by Eick and  Wright and \cite{HOFLING20014} by H\"ofling  and in the corresponding to them   GAP packages ``FORMAT'' \cite{Eick2000} and ``CRISP'' \cite{Hoefling2000} respectively. These papers are dedicated to finding $\mathfrak{F}$-projectors, $ \mathfrak{F}$-injectors, $\mathfrak{F}$-residuals and $ \mathfrak{F}$-radicals of soluble groups.  Also algorithms for classes of groups where permutability (or one of its generalizations) of subgroups is a transitive relation are studied in the paper \cite{BallesterBolinches2013} by  Ballester-Bolinches,  Cosme-Ll\'opez and Esteban-Romero and in corresponding to it GAP package ``permut'' \cite{BallesterBolinches2014}. The principal novelty of this paper is our ability to deal with non-saturated (non-local) classes  of not necessary soluble groups.

A finite group can be defined in the different ways. The most known of them are defining group by presentation,  permutations or matrices. One of the main  results in the foundation of the theory of formations of finite groups is Sylow's theorems. In \cite{Kantor1985} Kantor proved that a Sylow subgroup of a permutation group of degree $n$ can be found in polynomial time of $ n$ (mod CSFG). So it is natural to ask the following question:

\begin{pr}
  For a given class of groups $\mathfrak{X}$ and a permutation group $G$ of degree $n$ is there a polynomial-type algorithm that checks wether $G$ belongs to $\mathfrak{X}$?\end{pr}

That is why we introduce the following definition:

\begin{defn}
  We shall call a class of groups $\mathfrak{X}$ $P$-recognizable if for every $K\trianglelefteq G\leq S_n$ there is a polynomial-time algorithm that  tests wether $ G/K$ belongs  $\mathfrak{X}$ or not.
\end{defn}


Recall that a formation is a class of groups closed under taking homomorphic images and subdirect products.  One of the classical ways to study the structure of a group is to find  the action of a group on its chief series. For example,  formations of nilpotent, supersoluble and quasinilpotent groups; rank \cite[VII, Definitions 2.3]{Doerk1992}, local \cite[IV, Definitions 3.1]{Doerk1992},  Baer-local \cite[IV, Definitions 4.9]{Doerk1992} and graduated (see \cite[\S3]{Shemetkov1978} or \cite[\S 5.5]{Guo2015}) formations are defined by the action of a group on its chief factors.
All these formations are particular cases of the following construction.

\begin{defn}\label{def1}
Let $\mathrm{f}$ be a function which assigns 0 or 1 to every group $G$ and its chief factor $H/K$ such that

(1) $\mathrm{f}(H/K, G)=\mathrm{f}(M/N, G)$ whenever $H/K$ and $M/N$ are $G$-isomorphic chief factors of $G$;

(2) $\mathrm{f}(H/K, G)=\mathrm{f}((H/N)/(K/N), G/N)$ for every $N\trianglelefteq G$ with $N\leq K$.\\
Such functions $\mathrm{f}$ will be called chief factor functions.
Denote by $\mathcal{C}(\mathrm{f})$ the class of groups
\begin{center}
  $(G\mid \mathrm{f}(H/K, G)=1$ for every chief factor $H/K$ of a group group $G)$.
\end{center}
\end{defn}

For every non-empty formation $\mathfrak{F}$  in every group $G$ there exists the $\mathfrak{F}$-residual of $G$, that is the smallest normal subgroup  $G^\mathfrak{F}$  of $G$ with $G/G^\mathfrak{F}\in\mathfrak{F}$. It is clear that $\mathfrak{F}=(G\mid G^\mathfrak{F}=1)$.

\begin{thm}\label{residual}
  Assume that $\mathrm{f}(H/K, G)$ can be computed in polynomial time for every group $G$ and its chief factor $H/K$. Then  $\mathfrak{F}=\mathcal{C}(\mathrm{f})$ is a  $P$-recognizable formation and $G^\mathfrak{F}$ can be computed in polynomial time for   every $ G\leq S_n$.
\end{thm}



The concept of subnormality plays an important role in the group's theory. The formational generalization of this concept  was introduced in the universe of soluble groups by Hawkes \cite{Hawkes1969} and in the universe of all groups by Shemetkov (see \cite[Definition 8.1]{Shemetkov1978}). A subgroup $H$ of $G$ is called  $\mathfrak{F}$-subnormal in $G$, if $H=G$ or there exists a maximal chain of subgroups   $H=H_0\subset H_1\subset\dots\subset H_n=G$ such that $H_{i}/\mathrm{Core}_{H_{i}}(H_{i-1})\in\mathfrak{F}$ for $i=1,\dots,n$.
 Note that if $\mathfrak{F}$ is a hereditary formation, then the word
 ``maximal'' can be omitted in this definition.
  Kegel \cite{Kegel1978} introduced the another such  generalization of subnormality. Recall \cite[Definition 6.1.4]{BallesterBollinches2006} that a subgroup $H$ of  $G$ is called $K$-$\mathfrak{F}$-\emph{subnormal} in $G$ if there is a chain of subgroups
$ H=H_0\subseteq H_1\subseteq\dots\subseteq H_n=G$
with $H_{i-1}\trianglelefteq H_i$ or $H_{i}/\mathrm{Core}_{H_{i}}(H_{i-1})\in\mathfrak{F}$ for all $i=1,\dots,n$.   If $\mathfrak{F}=\mathfrak{N}$, then the notions of  $K$-$\mathfrak{F}$-subnormal and subnormal subgroups coincide.
For more information about $\mathfrak{F}$-subnormal and $K$-$\mathfrak{F}$-subnormal subgroups see
\cite[Chapter 6]{BallesterBollinches2006}.

\begin{thm}\label{Fsubnorm}
  Let $\mathfrak{F}$ be a hereditary formation. Assume that $G^\mathfrak{F}$ can be computed in polynomial time for   every $ G\leq S_n$ and natural $n$. Then there are polynomial-time algorithms that tests wether given subgroup is $\mathfrak{F}$-subnormal or $K$-$\mathfrak{F}$-subnormal.
\end{thm}

\section{Preliminaries}

\subsection{Groups and their classes}

Recall that $M^G$ denotes the smallest normal subgroup of $G$ which contains $M$; $M'$ is the derived subgroup of $M$; $M^p$ is the subgroup generated by $p$-th powers of elements of $M$; $S_n$ denotes\,the symmetric group on $n$ elements; a formation $\mathfrak{F}$ is called hereditary, if $H\leq G\in\mathfrak{F}$ implies\,$H\in\mathfrak{F}$.   

The material of this section can be found, for example, in \cite[p. 5-8]{Doerk1992}.   Let $\Omega$  be a set. A group $G$ is called an $\Omega$-\emph{group} if there is associated with each element $\omega\in\Omega$ an endomorphism of $G$ denoted for all $g\in G$ by
$g\rightarrow g\omega$. A subgroup $U$ of $G$ is called $\Omega$-\emph{admissible} if $u\omega\in U$ for all $u\in U$ and $\omega\in\Omega$.
Evidently the intersection and the join of $\Omega$-admissible subgroups are again
$\Omega$-admissible.
If $N$ is an $\Omega$-admissible normal subgroup of $G$, the quotient group $G/N$
may be regarded naturally as an $\Omega$-group via the action defined for all $g\in G$ and $\omega\in\Omega$ by
$(Ng)\omega=N(g\omega)$.
   Finally if $G$ and $H$ are $\Omega$-groups, a isomorphism $\alpha: G \rightarrow H$ is called an $\Omega$-\emph{isomorphism} if for all $g\in G$ and $\omega\in\Omega$
holds $\alpha(g\omega)=\alpha(g)\omega$.

\begin{thm}[The Isomorphism Theorems]\label{th1} Let $\Omega$ be a set and let $G$  be an  $\Omega$-group.

$(1)$ If $U$ and $N$ are  $\Omega$-admissible subgroups of $G$ and $U$ normalizes $N$,
then $UN/N \simeq U/(U \cap N)$ as $\Omega$-groups.

$(2)$   If $M$ and $N$ are $\Omega$-admissible normal subgroups of $G$ and $N < M$,
then the $\Omega$-groups $(G/N)/(M/N)$ and $G/M$ are $\Omega$-isomorphic.
\end{thm}

An $\Omega$-group is called $\Omega$-\emph{simple} if 1 and $G$   are the only $\Omega$-admissible
normal subgroups of $G$. A a subnormal chain $U=U_0, U_1, \dots, U_n=G$   from $U$ to $G$  is called $\Omega$-\emph{series} if   all of its terms are $\Omega$-admissible.
An $\Omega$-series is called $\Omega$-\emph{composition series} if each factor $U_i/U_{i-1}$ is $\Omega$-simple for $i=1,\dots, n$.

\begin{thm}[The Jordan-H\"older Theorem]\label{th2}  Let $G$ be an $\Omega$-group, and let
\[1=N_0\vartriangleleft N_1\vartriangleleft\dots\vartriangleleft N_n=G
\hspace{3mm} and \hspace{3mm} 1=M_0\vartriangleleft M_1\vartriangleleft\dots\vartriangleleft M_m=G\]
be two $\Omega$-composition series of $G$. Then $m =n$ and there exists a permutation $\pi\in S_n$
such that for $i =1,\dots, n$ the factor $N_i/N_{i-1}$ is $\Omega$-isomorphic with $M_{\pi(i)}/M_{\pi(i)-1}$.
\end{thm}

\subsection{Computational conventions}

Here we use standard computational conventions of abstract finite groups equipped
with poly\-nomial-time procedures to compute products and inverses of elements (for the related
abstract notion of black-box groups, see \cite[Chapter 2]{Seress2003}).

Unless stated
otherwise, for both input and output, groups are specified by generators. We will consider only $G=\langle S\rangle\leq S_n$ with $|S|\leq n^2$. If necessary, Sims' algorithm \cite[Parts 4.1 and 4.2]{Seress2003} can be used to arrange that $|S|\leq n^2$.

Quotient groups are specified by generators of a group and its normal subgroup.

According to \cite{Babai1986} the following result, all subgroups chains have  the bounded length:

\begin{lem}[\cite{Babai1986}]\label{chain}
   Given $G \leq S_n$ every chain of subgroups of $G$ has at most $2n-3$  members for $n\geq 2$.
\end{lem}

We need the following well known basis tools in our proves (see, for example \cite{Kantor1990a} or \cite{Seress2003}).

\begin{thm}\label{Basic}
  Given $G = \langle S\rangle\leq S_n$, in polynomial time one can solve the following problems:

  \begin{enumerate}
    \item  Find $|G|$.

    \item Given normal subgroups $A$ and $B$ of $G$,   find a composition series for $G$ containing them.

    \item Given $T\subseteq G$ find $\langle T\rangle^G$.

    \item (mod CFSG) Given $N, K \leq S_n$ such that $N/K$ is normalized by $G/K$,
     find $C_{G/K}(N/K)$ \cite[P6(i)]{Kantor1990a}.

    \item (mod CFSG) Given a prime $p$ dividing $|G|$, find a Sylow $p$-subgroup $P$ of $G$ and $N_G(P)$ \cite{Kantor1990}.

\item Given   $H=\langle S_1\rangle, K=\langle S_2\rangle \leq G$ find $\langle H, K\rangle=\langle S_1, S_2\rangle$ and $[H, K]=\langle \{[s_1, s_2]\mid s_1\in S_1, s_2\in S_2\rangle^{\langle H, K\rangle}$.


  \end{enumerate}
\end{thm}

\begin{lem}[{\cite[p. 155]{Seress2003}}]\label{transform}
Let $H$ and $K$ be  normal subgroups of $G$ such that $H/K$ is an elementary abelian $p$-group for some prime $p$. Then $H/K$ can be considered as $\mathbb{F}_pG$-module.  Every generator of $G$  induces by conjugation an linear transformation of this module. Its   matrix  can be computed in a polynomial time.
\end{lem}

\section{Proves of the Main Results}

\subsection{Proof of Theorem \ref{residual}}

The first step is to prove 

\begin{lem}
 If $ \mathrm{f}$ is a chief factor function, then $\mathcal{C}(\mathrm{f})$ is a formation.
\end{lem}

\begin{proof} Let $G\in\mathcal{C}(\mathrm{f})$ and $N\trianglelefteq G$. Then if $(H/N)/(K/N)$ is a chief factor of $G/N$, then $H/K$ is a chief factor of $ G$ and $ \mathrm{f}((H/N)/(K/N), G/N)=\mathrm{f}(H/K, G)=1$ by $ (2)$ of Definition \ref{def1}. Hence $G/N\in \mathcal{C}(\mathrm{f})$. It means that $\mathcal{C}(\mathrm{f})$ is closed under taking homomorphic images.  Assume now $G/N, G/M\in\mathcal{C}(\mathrm{f})$ and $ M\cap N=1$.
Let $ H/K$ be a chief factor of $G$ below $N$. Then $$HM/KM\simeq H/(H\cap KM)=H/K(H\cap M)=H/K,$$  i.e. $H/K$ is $ G$-isomorphic to a chief factor of $G$ above $M$ by $(1)$ of The Isomorphism Theorems. From the Jordan-H\"older Theorem it follows that every chief factor of $G$ is $G$-isomorphic to a chief factor of $G$ above $M$ or $N$. WLOG let $H/K\simeq R/T$ and $ R/T$ is a chief factor of $G$ above $N$, then $\mathrm{f}(H/K, G)=\mathrm{f}(R/T, G)=\mathrm{f}((R/N)/(T/N), G/N)=1$ by Definition \ref{def1}. It means that $G\in\mathcal{C}(\mathrm{f})$. Hence $\mathcal{C}(\mathrm{f})$ is closed under taking subdirect products. It means that $ \mathcal{C}(\mathrm{f})$ is a formation.
\end{proof}

Recall that the smallest normal subgroup $H$ of $G$ such that $G/H$ is the direct product of simple (resp. simple non-abelian) subgroups of $G$ is called the (resp. non-abelian) residual of $G$ and is denoted by $\mathrm{Res}(G)$  (resp. $\mathrm{Res}_N(G)$). Here we are interested in the following subgroups. Let $M$ be a normal subgroup of a group $G$. Denote by $\mathrm{Res}_N(M, G)$ (resp. $\mathrm{Res}_p(M, G)$)   the smallest normal subgroup $H$ of $G$ below $M$ such that $M/H$ is the direct product of  minimal normal  non-abelian (resp. $p$-subgroups) subgroups    of $G/H$.

\begin{lem}\label{lem2}
  $\mathrm{Res}_N(N, G)$ is defined for every normal subgroup $N$ of $G$. Moreover $\mathrm{Res}_N(N, G)$ and a decomposition of $N/\mathrm{Res}_N(N, G)$ into the direct product of minimal normal subgroups of $G/\mathrm{Res}_N(N, G)$ can be computed in a polynomial time.
\end{lem}

\begin{proof}
Note that $\mathrm{Res}_N(N)\textrm{ char }N\trianglelefteq G$. Hence $\mathrm{Res}_N(N)\trianglelefteq G$. Recall that $\mathrm{Res}_N(N)$ is the smallest normal subgroup of $N$ such that $N/\mathrm{Res}_N(N)$ is a direct product of simple non-abelian groups and every minimal normal non-abelian subgroup is a direct product of simple non-abelian groups.
 Therefore if $ \mathrm{Res}_N(N, G)$ exists, then  $\mathrm{Res}_N(N)$ contains it. Let prove that $N/\mathrm{Res}_N(N)$ is the direct product of minimal normal non-abelian subgroups of $G/\mathrm{Res}_N(N)$.

 Let $A=N$
and $M/\mathrm{Res}_N(N)$ be a simple subnormal subgroup of $A/\mathrm{Res}_N(N)$. Then $M/\mathrm{Res}_N(N)$ is a simple non-abelian subnormal subgroup of $G/\mathrm{Res}_N(N)$. Therefore \linebreak $(M/\mathrm{Res}_N(N))^G$ is a  minimal normal subgroup of $G/\mathrm{Res}_N(N)$ below $A/\mathrm{Res}_N(N)\leq N/\mathrm{Res}_N(N)$. Note that $$ A/\mathrm{Res}_N(N)=(M/\mathrm{Res}_N(N))^G\times C_{A/\mathrm{Res}_N(N)}((M/\mathrm{Res}_N(N))^G),$$  $C_{A/\mathrm{Res}_N(N)}((M/\mathrm{Res}_N(N))^G)\trianglelefteq G/\mathrm{Res}_N(N)$. So now we can let $$A/\mathrm{Res}_N(N)\leftarrow C_{A/\mathrm{Res}_N(N)}((M/\mathrm{Res}_N(N))^G).$$

 It means that using the previous steps we can decompose $N/\mathrm{Res}_N(N)$ into the direct product  of minimal normal  non-abelian subgroups of $G/\mathrm{Res}_N(N)$. Thus $\mathrm{Res}_N(N, G)=\mathrm{Res}_N(N)$.


\begin{algorithm}[H]
\caption{NonAbelianDecomposition($G, N$)}
\SetAlgoLined
\KwResult{The smallest normal subgroup $K$ of $G$ below $N$ such that $N/K\simeq \overline{N}_1\times\dots\times\overline{N}_k$ where  $\overline{N}_i$ is a minimal normal non-abelian subgroup of $G$;  subgroups $\overline{N}_i$. }
\KwData{$N$ is a normal subgroup of a group $G$}
$K\gets Res_N(N)$\;

 $A\gets N$\;

 $L\gets []$\;

\While{$|A|\neq |K|$}{
Find a minimal subnormal subgroup $B/K$ of $A/K$\;
Find $(B/K)^G$ and add this subgroup to $L$\;
$A/K\gets C_{A/K}((B/K)^G)$\;
}
 \end{algorithm}

According to \cite[Theorem 8.3]{Babai1987}  $\mathrm{Res}_N(N)$ can be found in polynomial time. By 2 of Theorem \ref{Basic} the minimal subnormal subgroup $M/\mathrm{Res}_N(N)$ of $N/\mathrm{Res}_N(N)$ can be found in a polynomial time. Then $(M/\mathrm{Res}_N(N))^G=M^G/\mathrm{Res}_N(N)$  can be computed in a polynomial time by 3 of Theorem \ref{Basic}. Now  $C_{A/\mathrm{Res}_N(N)}((M/\mathrm{Res}_N(N))^G)$  can be computed in a polynomial time by 4 of Theorem \ref{Basic}.
Thus Algorithm 1 runs in  a polynomial time by Lemma \ref{chain}.
\end{proof}

\begin{lem}\label{lem3}
Let $p$ be a prime.   $\mathrm{Res}_p(N, G)$ is defined for every normal subgroup $N$ of $G$. Moreover $\mathrm{Res}_p(N, G)$ and a decomposition of $N/\mathrm{Res}_p(N, G)$ into the direct product of minimal normal subgroups of $G/\mathrm{Res}_p(N, G)$ can be computed in a polynomial time.
\end{lem}

\begin{proof}
  Let $N/K$ be the direct product of minimal normal  $p$-subgroups of $G/K$  for a given $p$ where $N, K\trianglelefteq G$. Note that in this case $N/K$ is the elementary abelian $p$-group.
  It means that $N'N^p\subseteq K$. So if $N'N^p=N$, then we can let $\mathrm{Res}_p(N, G)=N'N^p$. Assume that $N\neq N'N^p$. Note that $N'N^p\textrm{ char }N\trianglelefteq G$. Hence $N'N^p\trianglelefteq G$.   Then we can consider $V=N/N'N^p$ as an $\mathbb{F}_pG$-module.

In this case normal subgroups of $G/N'N^p$ below $N/N'N^p$ are in the one to one correspondence with submodules of $V$. Let $K/N'N^p$ be the   radical $\mathrm{Rad}(N/N'N^p)$ of $N/N'N^p$. Now $N/K\simeq (N/N'N^p)/\mathrm{Rad}(N/N'N^p)$ is a semisimple $ \mathbb{F}_pG$-module, i.e. $N/K$ is the direct product of minimal normal subgroups of $ G/K$.

Assume that $ K_1$ is a normal subgroup of $ G$ such that $N/K_1=N_1/K_1\times\dots\times N_k/K_1$ is   a direct product of minimal normal $ p$-subgroups $ N_i/K_1$ of $ G/K_1$. It is clear that $ N'N^p\subseteq K_1$. Note that $ N/(\prod_{i=1,i\neq j}^n N_i)$ is a chief factor of $G$. It means that $\prod_{i=1,i\neq j}^n N_i/N'N^p$ is a maximal submodule of $V$. Recall that the   radical of a module is the intersection of all its maximal submodules. Now $$K/N'N^p=\mathrm{Rad}(N/N'N^p)\subseteq\bigcap_{j=1}^n(\prod_{i=1,i\neq j}^n N_i/N'N^p)=K_1/N'N^p. $$ Thus $K\subseteq K_1$. It means that $K$ is the smallest normal subgroup   $G$ below $N$  such that $N/K$ is the direct product of minimal normal $p$-subgroups  of $G/K$. Hence $K=\mathrm{Res}_p(N, G)$ is well defined.

Let show that $K$ can be computed in polynomial time. If $N=\langle S\rangle$, then $N'N^p=\langle\{[x,y]\mid x,y\in S\}\cup\{x^p\mid x\in S\}\rangle$ can be computed in polynomial time. Every generator of $G$ induces by conjugation a linear transformation of $N/N'N^p$. The matrix of this transformation can be found in a polynomial time by Lemma \ref{transform}. Denote the algebra generated by these matrixes by $R$. Then the basis of its Jacobson radical $J(R)$ can be computed in a polynomial time by \cite[Theorem 2.7]{Ronyai1990}. Now $\mathrm{Rad}(N/N'N^p)=(N/N'N^p)J(R)$ by  \cite[B, Proposition 4.2]{Doerk1992}. Hence $\mathrm{Rad}(N/N'N^p)$ is generated (as a module and as a subgroup) by  products $ nr$ where $n$ is a generator $N/N'N^p$ and $r$ is a generator of $J(R)$. Thus $\mathrm{Rad}(N/N'N^p)$ can be computed in a polynomial time, i.e. $\mathrm{Rad}(N/N'N^p)=K/N'N^p$ and we know generators of $K$ as a subgroup of $G$.

Since every generator of $G$ induces by conjugation a linear transformation of $N/K$, the matrix of this transformation can be found in a polynomial time by Lemma \ref{transform}.
Denote the algebra generated by these matrixes by $R$.
Note that $N/K$ is a semisimple $\mathbb{F}_pG$-module. Hence it is a semisimple $R$-module. Now $nr=0$ for every $n\in N/K$ and $r\in J(R)$. Since $R$ acts on $N/K$ by matrix multiplications, we see that $J(R)=0$. Thus $R$ is semisimple. Then bases of minimal ideals $R_i$  of $R$ can be found in polynomial time by  \cite[Corollary 3.2]{Ronyai1990}. Then $(N/K)R_i$ is a minimal submodule of $N/K$ and the sum of this submodules is $N/K$  by \cite[VII, Theorem 12.1]{Huppert1982}. It is clear that generating sets of these submodules (and hence corresponding to them quotient groups) can be found in a polynomial time.


\begin{algorithm}[H]
\caption{PDecomposition($G, N, p$)}
\SetAlgoLined
\KwResult{The smallest normal subgroup $K$ of $G$ below $N$ such that $N/K\simeq \overline{N}_1\times\dots\times\overline{N}_k$ where  $\overline{N}_i$ is a minimal normal $p$-subgroup of $G$;  subgroups $\overline{N}_i$. }
\KwData{$N$ is a normal subgroup of a group $G$ and $ p$ is a prime}
$M\gets[]$\;
$ L\gets[]$\;
\If{$|N^pN'|=|N|$}{output $ N^pN'$ and $L$\;}
\Else{For each generator $ g$ of $G$ find the linear transformation which this element induces on $N/N^pN'$\;
For the algebra generated by above mentioned transformations $R$ find the basis of   $\mathrm{J}(B)$\; 
Find the generators of $K$ where $K/N'N^p=(N/N'N^p)J(R)$\;
For each generator $ g$ of $G$ find the linear transformation which this element induces on $N/K$\;
Decompose the algebra $R$ generated by above mentioned transformations  into the sum $\rho_1\oplus\dots\oplus\rho_k$  of minimal left ideals\;
For each ideal $ \rho_i$ find $(N/K)\rho_i$ and add it to $ M$\;
For each element in $M$ find its generators in $G$ and add them as an element to $L$\;}
 \end{algorithm}
\end{proof}


\begin{lem}\label{lem4}
    Let $\mathfrak{F}=\mathcal{C}(\mathrm{f})$, $N$ and $K$ be  normal subgroups of $G$ such that $N/K=N_1/K\times\dots\times N_t/K$ where $N_i/K$ is a minimal normal subgroup of $G$ and $G/N\in\mathfrak{F}$. Then $ (G/K)^\mathfrak{F}$ can be found in polynomial time.

  \end{lem}

\begin{proof}
  Let $$I^+=\{i\mid \mathrm{f}(N_i/K, G)=1\}, I^-=\{i\mid \mathrm{f}(N_i/K, G)=0\} \textrm{ and } M/K=\prod_{i\in I^-}  N_i/K.$$
We claim that $M/K=(G/K)^\mathfrak{F}$.  Note that every chief factor $H/T$  of $G$ between $M$ and $N$ is $G$-isomorphic to $N_i/K$ for  some $i\in I^+$.  Hence   $$\mathrm{f}((H/M)/(T/M), G/M)=\mathrm{f}(H/T, G)=\mathrm{f}(N_i/K, G)=1.$$ Since $G/N\simeq (G/M)/(N/M)\in\mathcal{C}(\mathrm{f})$, we see that $\mathrm{f}((H/M)/(T/M), G/M)=1$ for every chief factor   $(H/M)/(T/M)$  of $G/M$ above $N/M$. From the Jordan-H\"older theorem it follows that $(G/K)/(M/K)\simeq G/M\in \mathcal{C}(\mathrm{f})=\mathfrak{F}$. Hence $(G/K)^\mathfrak{F}\leq M/K$.

Assume that $F/K=(G/K)^\mathfrak{F}< M/K$. So $I^-\neq\emptyset$. Then $F/K< FN_i/K$ for some   $i\in I^-$, i.e. $ F\cap N_i= K$.   Hence $FN_i/F$ and $N_i/K$ are $G$-isomorphic chief factors of $G$.     Thus \begin{multline*}
  1=\mathrm{f}(((FN_i/K)/(F/K))/((F/K)/(F/K)), (G/K)/(F/K))=\mathrm{f}((FN_i/K)/(F/K), G/K)=\\
  \mathrm{f}(FN_i/F, G)=\mathrm{f}(N_i/K, G)=0,
\end{multline*} a contradiction. Thus $(G/K)^\mathfrak{F}= M/K$.

\begin{algorithm}[H]
\caption{FResidualPart($G, N, K, L, \mathrm{f}$)}
\SetAlgoLined
\KwResult{$T/K=(G/K)^\mathfrak{F}$. }
\KwData{$N\trianglelefteq G$ with $G/N\in\mathfrak{F}$, $K\trianglelefteq G$ with $N/K=N_1/K\times\dots\times N_t/K$; the list $L$ of minimal normal subgroups $N_i/K$ of $G/K$.}
$T\gets K$\;
\For{$i$ in $[1,..., t]$}
{\If{$\mathrm{f}(N_i/K, G)=0$}{$T\gets\langle T, N_i\rangle$}}
 \end{algorithm}

Since $\mathrm{f}(H/K, G)$ can be computed in a polynomial time for every chief factor $H/K$ of $G$, we see that      $I^-$ can be computed in a polynomial time. Note that $t<2n$ by Lemma \ref{chain}. Hence the join of  not more than $t$ subgroups can be computed in a polynomial time.
\end{proof}

\begin{lem}\label{lem7}
  Let $\mathfrak{F}=\mathcal{C}(\mathrm{f})$ and $G$ be a group. Then $G^\mathfrak{F}$ can be computed in a polynomial time.
\end{lem}

\begin{proof}
  Note that $G/G\in\mathcal{C}(\mathrm{f})$. Assume that we have a subgroup $H$ with $G/H\in \mathcal{C}(\mathrm{f})$. Then $G^\mathfrak{F}\subseteq H$. If $G^\mathfrak{F}\neq H$, then there is a chief factor $H/K$ of $G$ with $\mathrm{f}(H/K, G)=1$. Note that $H/K$ is either a non-abelian or an abelian $p$-group. Hence $H/K$ is $G$-isomorphic to a chief factor of $G$ between $\mathrm{Res}_N(H, G)$ and $H$ or   between $\mathrm{Res}_p(H, G)$ and $H$ for some $p$.

\begin{algorithm}[H]
\caption{FResidual($G, \mathrm{f}$)}
\SetAlgoLined
\KwResult{$N=G^\mathfrak{F}$. }
\KwData{$\mathfrak{F}=\mathcal{C}(\mathrm{f})$, $G$ is a group.}

$K\gets G$\;
 \Repeat{$|N|\neq|K|$}{
 $N\gets K$\;
 $K\gets$FResidualPart($G$, $K$, NonAbelianDecomposition($G, K$), $\mathrm{f}$)\;
 $\pi\gets\pi(K)$\;
\For{$p$ in $\pi$}
{$K\gets$FResidualPart($G$, $K$, PDecomposition($G, K, p$), $\mathrm{f}$)\;}
}
 \end{algorithm}
From Lemmas \ref{lem2}--\ref{lem7} it follows that this is a polynomial time algorithm.
\end{proof}

\begin{lem}
  Let $\mathfrak{F}=\mathcal{C}(\mathrm{f})$. Then $\mathfrak{F}$ is $P$-recognizable.
\end{lem}

\begin{proof}
  Let $G$ be a group. Then $G^\mathfrak{F}$ can be computed in  a polynomial time. Recall that $(G/K)^\mathfrak{F}=G^\mathfrak{F}K/K$. Hence $G/K\in\mathfrak{F}$ iff $G^\mathfrak{F}\subseteq K$ iff $\langle G^\mathfrak{F}, K\rangle=K$ iff  $|\langle G^\mathfrak{F}, K\rangle|=|K|$. The last condition can be checked in  polynomial time by 1 and 6 of Theorem \ref{Basic}.
\end{proof}

\section{Proof of Theorem \ref{Fsubnorm}}

Let $H$ be a $K$-$\mathfrak{F}$-subnormal subgroup of  $G$. From the definition of $K$-$\mathfrak{F}$-subnormal subgroup it follows that   either $G=H$ or there is a proper subgroup $M$ of $G$ with $H$ is a $K$-$\mathfrak{F}$-subnormal subgroup of  $M$ and $M\trianglelefteq G$ or $G^\mathfrak{F}\leq M$.

\begin{algorithm}[H]
\caption{ISKFSUBNORMAL$(G, H, \mathfrak{F})$}
\SetAlgoLined
\KwResult{True if $H$ is $K$-$\mathfrak{F}$-subnormal in $G$ and False otherwise.}
\KwData{A subgroup $H$ of a  group $G$.}
\eIf{$|G|=|H|$}{{\bf return} True;}
{\eIf{$|HG^\mathfrak{F}|\neq|G|$}{{\bf return}  ISKFSUBNORMAL$(HG^\mathfrak{F}, H, \mathfrak{F})$;}
{\If{$|H^G|\neq|G|$}{{\bf return}  ISKFSUBNORMAL$(H^G, H, \mathfrak{F})$;}}{{\bf return} False;}}
 \end{algorithm}

Since $G^\mathfrak{F}$ can be computed in a polynomial time and according to 1 and 3 of Theorem \ref{Basic}, we see that every above mentioned check can be made in a polynomial time. Now the statement of theorem follows from the fact that every chain of subgroups of $G$ has at most $2n$  members by Lemma \ref{chain}.

By analogy one can prove that the following algorithm tests $\mathfrak{F}$-subnormality in a polynomial time.

\begin{algorithm}[H]
\caption{ISFSUBNORMAL$(G, H, \mathfrak{F})$}
\SetAlgoLined
\KwResult{True if $H$ is $\mathfrak{F}$-subnormal in $G$ and False otherwise.}
\KwData{A subgroup $H$ of a  group $G$.}
\eIf{$|G|=|H|$}{{\bf return} True;}
{\eIf{$|HG^\mathfrak{F}|\neq|G|$}{{\bf return}  ISFSUBNORMAL$(HG^\mathfrak{F}, H, \mathfrak{F})$;}
{{\bf return} False;}}
 \end{algorithm}

\section{Applications}

\subsection{Local and Baer-local formations}

Recall \cite[IV, Definitions 3.1]{Doerk1992} that a function $f$ which assigns a formation to each prime    is called a \emph{formation function}; a chief factor $H/K$ of a group $G$ is called \emph{$f$-central} if $G/C_G(H/K)\in f(p)$ for all prime divisors of $|H/K|$; a formation $\mathfrak{F}$ is called \emph{local} if its coincides with the class of groups all whose chief factors are $f$-central for some formation function $f$. In this case $f$ is called a local definition of $\mathfrak{F}$.

\begin{thm}\label{local}
  Let $f_\mathfrak{F}$ be a local definition of a local formation $\mathfrak{F}$. Assume that $G^{f_\mathfrak{F}(p)}$ can be computed in a polynomial time for every prime $p$ and every group $G$. Then $\mathfrak{F}$ is $P$-recognizable and $G^\mathfrak{F}$ can be computed in a polynomial time.
\end{thm}

\begin{proof} Note that
\begin{align*}
  &G/C_G(H/K)\in f_\mathfrak{F}(p) \qquad\quad\,\,\forall p\in\pi(H/K)\\
  &\Leftrightarrow G^{f_\mathfrak{F}(p)}\subseteq C_G(H/K) \quad\quad\,\,\,\,\forall p\in\pi(H/K)\\
  &\Leftrightarrow [G^{f_\mathfrak{F}(p)}, H]\subseteq K \qquad\qquad\,\,\forall p\in\pi(H/K)\\
  &\Leftrightarrow |\langle [G^{f_\mathfrak{F}(p)}, H], K\rangle|=|K| \,\,\,\,\,\forall p\in\pi(H/K).
  \end{align*}
Let
\begin{displaymath}
\mathrm{f}_{\mathfrak{F}}(H/K, G)=\begin{cases}
  1,& H/K \textrm{ is }f_\mathfrak{F}\textrm{-central};\\
  0,& \textrm{ otherwise}.
\end{cases}
=\begin{cases}
  1,& |\langle [G^{f_\mathfrak{F}(p)}, H], K\rangle|=|K|\quad\forall p\in\pi(H/K);\\
  0,& \textrm{ otherwise}.
\end{cases}
\end{displaymath}
  From the definition of local formation it follows that  $\mathfrak{F}=\mathcal{C}(\mathrm{f}_\mathfrak{F})$. Since $G^{f_\mathfrak{F}(p)}$, the commutator of two subgroups, the join of two subgroups and the order of subgroup can be computed in a polynomial time, we see that $\mathrm{f}_{\mathfrak{F}}(H/K, G)$ can be computed in a polynomial time.

Lets prove that $\mathrm{f}_\mathfrak{F}$ is a chief factor function.
  If $H/K$ and $M/N$ are $G$-isomorphic chief factors, then $C_G(H/K)=C_G(M/N)$. Hence $G/C_G(H/K)=G/C_G(M/N)$. Therefore $\mathrm{f}_\mathfrak{F}$ satisfies (1) of Definition \ref{def1}.
Note that if $[G^{f_\mathfrak{F}(p)}, H]\subseteq K$ for all $p\in\pi(H/K)$, then
$$[(G/N)^{f_\mathfrak{F}(p)}, H/N]=[G^{f_\mathfrak{F}(p)}N/N, H/N]=[G^{f_\mathfrak{F}(p)}, H]N/N\subseteq K/N \quad\forall p\in\pi((H/N)/(K/N)).$$ Hence  $\mathrm{f}_\mathfrak{F}$ satisfies (2) of Definition \ref{def1}.

Therefore the statement of Theorem \ref{local} directly follows from Theorem \ref{residual}.
\end{proof}


The following classes of groups are local formations:
\begin{enumerate}
\item The class $\mathfrak{U}$ of all supersoluble groups. It is locally defined by $f_\mathfrak{U}(p)=\mathfrak{A}(p-1)$ (the class of all abelian groups of exponent dividing $p-1$).

 \item The class $w\mathfrak{U}$  of widely supersoluble groups \cite{Vasilev2010}. It is  locally defined by $f_{w\mathfrak{U}}(p)=\mathcal{A}(p-1)$ (the class of all groups all whose Sylow subgroups are abelian of exponent dividing $p-1$).

\item  The class $\mathfrak{N}\mathcal{A}$ of groups $G$ such that all Sylow subgroups of $G/\mathrm{F}(G)$ are abelian \cite{Vasilev2010}. It is locally defined by  $f_{\mathfrak{N}\mathcal{A}}(p)=\mathcal{A}$ (the class of groups all whose Sylow subgroups are abelian).

\item In \cite{Zimmermann1989} the class $sm\mathfrak{U}$ of groups with submodular Sylow subgroups were studied. It is locally defined \cite{Vasilyev2015} by $f_{sm\mathfrak{U}}(p)=\mathcal{A}(p-1)\cap \mathfrak{B}$
where $\mathfrak{B}$ is a class of groups with square-free exponent.

\item The class of strongly supersoluble groups $s\mathfrak{U}$ \cite{Vasilyev2015}. Its local definition $f_{s\mathfrak{U}}(p)=\mathfrak{A}(p-1)\cap\mathfrak{B}$.

 \item  The class $sh\mathfrak{U}$ of groups all whose Schmidt subgroups are supersoluble \cite{Monakhov1995, Monakhov2021}.  Its local definition $f_{sh\mathfrak{U}}(p)=\mathfrak{G}_{\pi(p-1)}$ (the class of all $\pi(p-1)$-groups).
\end{enumerate}

\begin{cor}
  Let  $\mathfrak{F}\in\{\mathfrak{U}, \mathrm{w}\mathfrak{U}, s\mathfrak{U}, sm\mathfrak{U}, \mathfrak{N}\mathcal{A}, sh\mathfrak{U}\}$. Then $\mathfrak{F}$ is $P$-recognizable and
  $G^\mathfrak{F}$ can be computed in a polynomial type. In particular, $\mathfrak{F}$-subnormality of a subgroup can be tested in a polynomial time.  \end{cor}

\begin{proof}
Let $G=\langle S\rangle$. Note that the generating set $S_p$ of a Sylow $p$-subgroup of $G$ can be computed in a polynomial time by 4 of Theorem \ref{Basic}. It is straightforward to check that

1.  $G^{f_\mathfrak{U}(p)}=\langle \{[x, y], x^{p-1}\mid x, y\in S\} \rangle$.

2. $G^{f_{w\mathfrak{U}}(p)}=\langle (\bigcup_{p\in\pi(G)}\{[x, y], x^{p-1}\mid x, y\in S_p\})^G \rangle$.

3.  $G^{f_{\mathfrak{N}\mathcal{A}}(p)}=\langle (\bigcup_{p\in\pi(G)}\{[x, y]\mid x, y\in S_p\})^G \rangle$.

4.  $G^{f_{sm\mathfrak{U}}(p)}=\langle (\bigcup_{p\in\pi(G)}\{[x, y], x^{\prod_{q\in\pi(p-1)}q}\mid x, y\in S_p\})^G \rangle$.

5. $G^{f_{s\mathfrak{U}}(p)}=\langle \{[x, y], x^{\prod_{q\in\pi(p-1)}q}\mid x, y\in S\} \rangle$.

6. $G^{f_{sh\mathfrak{U}}(p)}=\langle (\bigcup_{p\not\in\pi(p-1)}\{x \mid x\in S_p\})^G \rangle$.

Hence every of the above mentioned subgroups can be computed in a polynomial time.  Note that all  these formations are hereditary. Thus the statement of a corollary directly follows from Theorems \ref{Fsubnorm} and \ref{local}.
\end{proof}

Let $f$ be a local definition of a local formation $\mathfrak{F}$.  Recall that if $f(p)\subseteq\mathfrak{F}$ for all $p$, then every $f$-central chief factor is called $\mathfrak{F}$-\emph{central} and every non-$f$-central chief factor is called  $\mathfrak{F}$-\emph{eccentric}.

\begin{lem}\label{Fcentral}
Assume that $f_\mathfrak{F}$ is a local definition of a local formation $\mathfrak{F}$ and    $G^{f_\mathfrak{F}(p)}$ can be computed in a polynomial time for every prime $p\in\pi(G)$. Then $G^{F_\mathfrak{F}(p)}$ can be computed in a polynomial time for every prime $p\in\pi(G)$ where $F_\mathfrak{F}$ is the canonical local definition of $\mathfrak{F}$. In particular the check of $\mathfrak{F}$-centrality of a chief factor can be done in a polynomial time.
\end{lem}

\begin{proof}
  Recall that $F_\mathfrak{F}(p)=\mathfrak{N}_p(\mathfrak{F}\cap f_\mathfrak{F}(p))$. So $G^{F_\mathfrak{F}(p)}=(G^{f_\mathfrak{F}(p)}G^\mathfrak{F})^{\mathfrak{N}_p}$ can be computed in a polynomial time   by Theorem \ref{local} for any $p\in\pi(G)$. Following the proof of this theorem we can check a chief factor for $F_\mathfrak{F}$-centrality (which is the same as $\mathfrak{F}$-centrality)   in a polynomial time.
\end{proof}

One of important families of formations are Baer-local or composition formations. There are many ways to define them (see \cite[IV, Definitions 4.9]{Doerk1992}, \cite[p. 4]{Guo2015} and \cite[Definition 3.11]{Shemetkov1978}).
 A function of the form $f: \{Simple\,\,groups\}\rightarrow\{formations\}$ is called a Baer function. $ f(Z_p)$ is denoted by $f(p)$ where $Z_p$ is a cyclic group of order $p$.
 A chief factor $H/K$ of a group $G$ is called \emph{$f$-central} if $G/C_G(H/K)\in f(S)$ where all composition factors of $H/K$ are isomorphic to $S$. A formation $\mathfrak{F}$ is called \emph{Baer-local} if its coincides with the class of groups all whose chief factors are $f$-central for some Baer function $f$.
It is known  (see \cite[1, Theorem 1.6]{Guo2015}) that Baer-local formation can be defined by Baer function $f$ such that $f(0)=f(S)$ for every non-abelian simple group.

\begin{thm}\label{composition}
  Let $f$ be a Baer-local definition of a  composition formation $\mathfrak{F}$. Assume that $G^{f(x)}$ can be computed in polynomial time for      every $x\in\mathbb{P}\cup\{0\}$. Then $\mathfrak{F}$  is $P$-recognizable and $G^\mathfrak{F}$ can be computed in polynomial time.
\end{thm}

\begin{proof} Let
  \begin{displaymath}
\mathrm{f}_{\mathfrak{F}}(H/K, G)=\begin{cases}
  1,& H/K \textrm { is non-abelian and } |\langle [G^{f_\mathfrak{F}(0)}, H], K\rangle|=|K|;\\
   1,&  H/K\textrm { is a $p$-group and } |\langle[G^{f_\mathfrak{F}(p)}, H], K\rangle|=|K|;\\
  0,& \textrm{ otherwise}.
\end{cases}
\end{displaymath}
As in the proof of Theorem \ref{local} we can chow that $\mathrm{f}_{\mathfrak{F}}$ is a chief factor function and $\mathfrak{F}=\mathcal{C}(\mathrm{f}_{\mathfrak{F}})$.
\end{proof}

\subsection{The lattice of chief factor functions}

For a chief factor functions $\mathrm{f}_1$ and $\mathrm{f}_2$ let
\begin{enumerate}
\item $(\mathrm{f}_1\vee\mathrm{f}_2)(H/K, G)=1$ iff $\mathrm{f}_1(H/K, G)=1$ or $\mathrm{f}_2(H/K, G)=1$.

\item $(\mathrm{f}_1\wedge\mathrm{f}_2)(H/K, G)=1$ iff $\mathrm{f}_1(H/K, G)=1$ and $\mathrm{f}_2(H/K, G)=1$.

\item $\overline{\mathrm{f}}_1(H/K, G)=1$ iff $\mathrm{f}_1(H/K, G)=0$.
\end{enumerate}
It is straightforward to check that these functions are chief factor functions. If $\mathrm{f}_1(H/K, G)$  and $\mathrm{f}_2(H/K, G)$ can be computed in a polynomial time, then functions from 1--3 can also be computed in a polynomial time. Note that the first two items defines the structure of a distributive lattice on the set of chief factor functions.

\begin{thm}
$P$-recognizable chief factor formations form a distributive lattice.
\end{thm}

The formation $\mathfrak{F}$ of groups whose $3$-chief factors are not central plays an important role  as a counterexample  in the study of mutual permutable products of groups (see \cite[Example 4.4.8]{PFG}). It is clear that this class of groups is defined by a chief factor function $\textrm{f}$ such that $\textrm{f}(H/K, G)=1$ if $H/K$ is not a central 3-chief factor or is   not a 3-chief factor. Since the orders of a chief factor and its centralizer can be computed in a polynomial time by Theorem \ref{Basic}, we see

\begin{prop}
  The formation of groups whose $3$-chief factors are not central is $P$-recognizable.
\end{prop}

It is well known that any Baer-local formation $\mathfrak{F}$ can be defined by Baer-function
 $F_\mathfrak{F}$ such that $F_\mathfrak{F}(0)=\mathfrak{F}$, i.e. the general definition of Baer-local
 formation gives little information about the action of an $\mathfrak{F}$-group $G$ on its non-abelian chief factors.
Therefore several families of Baer-local formations were introduced by giving additional
 information about the action of an $\mathfrak{F}$-group on its non-abelian chief factors.
For example, in \cite{Guo2009a, Guo2009}  Guo and Skiba introduced the class $\mathfrak{F}^*$ of  quasi-$\mathfrak{F}$-groups for a saturated formation $\mathfrak{F}$:

 \begin{center}
   $\mathfrak{F}^*=(G\,|$ for every $\mathfrak{F}$-eccentric chief factor $H/K$ and every $x\in G$, $x$ induces an inner automorphism on $H/K$).
 \end{center}
 If $\mathfrak{N}\subseteq\mathfrak{F}$ is a normally hereditary local formation, then $\mathfrak{F}^*$  is a normally hereditary Baer-local formation by \cite[Theorem~2.6]{Guo2009a}.

\begin{thm}\label{quasi}
   Let $f_\mathfrak{F}$ be a local definition of a local formation $\mathfrak{F}$. Assume that $G^{f_\mathfrak{F}(p)}$ can be computed in a polynomial time for every prime $p$. Then $\mathfrak{F}^*$ is $P$-recognizable and $G^{\mathfrak{F}^*}$ can be computed in a polynomial time.
\end{thm}

\begin{proof}
Note that every element of a group $G$  induces an inner automorphism on a chief factor $H/K$ if and only if $HC_G(H/K)=G$.  The last condition can be checked in a polynomial time by Theorem \ref{Basic}. Now we can check that either a chief factor is $\mathfrak{F}$-central or every element of a group $G$  induces an inner automorphism on it in a polynomial time by Lemma \ref{Fcentral}. Thus $\mathfrak{F}^*$ is $P$-recognizable and $G^{\mathfrak{F}^*}$ can be computed in a polynomial time by Theorem \ref{residual}.
\end{proof}

\begin{cor}
  Formation $\mathfrak{N}^*$ of all quasinilpotent groups is $P$-recognizable and $G^{\mathfrak{N}^*}$  can be computed in a polynomial time.
\end{cor}

\subsection{$\mathfrak{F}$-subnormal subgroups}

In   \cite{Monakhov2018,  Murashka2018, Semenchuk2011, Vasilev2016} groups with $K$-$\mathfrak{F}$-subnormal or $\mathfrak{F}$-subnormal Sylow subgroups were studied. The class of all groups with $K$-$\mathfrak{F}$-subnormal (resp. $\mathfrak{F}$-subnormal) Sylow $\pi$-subgroups was denoted by $\overline{w}_\pi\mathfrak{F}$ (resp. $w_\pi\mathfrak{F}$, see \cite{Vasilev2016}).

\begin{thm}
  Let $\mathfrak{F}$ be a hereditary  formation such that $G^\mathfrak{F}$ can be computed in a polynomial time for every group $G$  and  $\pi$ be a set of primes such that $ \pi(G)\cap\pi$ can be computed in polynomial time for every group $G$.   Then $\mathrm{w}_\pi\mathfrak{F}$ and $\mathrm{\overline{w}}_\pi\mathfrak{F}$  are  $P$-recognizable formations.
\end{thm}

\begin{proof}
  Note that if a Sylow $p$-subgroup of $G$ is $K$-$\mathfrak{F}$-subnormal, then every Sylow $p$-subgroup of every quotient group of $G$ is $K$-$\mathfrak{F}$-subnormal in it. Assume now that $P$ is not a $K$-$\mathfrak{F}$-subnormal Sylow subgroup of $G$. Then ISKFSUBNORMAL$(G, P, \mathfrak{F})$ finds a $K$-$\mathfrak{F}$-subnormal in $G$ subgroup $M=M(P)$ of $G$ with $P^M=M$ and $PM^\mathfrak{F}=M$. Assume that a Sylow $p$-subgroup $PN/N$ is $K$-$\mathfrak{F}$-subnormal in $G/N$.
  Since $\mathfrak{F}$ is a hereditary formation, we see that             $PN/N$ is $K$-$\mathfrak{F}$-subnormal in $MN/N$. From $(PN/N)^{MN/N}=P^MN/N=MN/N$  and $(PN/N)(M/N)^\mathfrak{F}=(PN/N)(M^\mathfrak{F}N/N)=MN/N$ it follows that $PN/N=MN/N$ is a $p$-group. Hence $\mathrm{O}^p(M)\leq N$.

  From the other hand $M$   is $K$-$\mathfrak{F}$-subnormal in $G$. Hence if $\mathrm{O}^p(M)\leq N$, then $PN/N=MN/N$ is     a $K$-$\mathfrak{F}$-subnormal Sylow $p$-subgroup of  $G/N$.

  Thus $G^{\mathrm{\overline{w}}_\pi\mathfrak{F}}$ is a normal closure of a subgroup generated by $\mathrm{O}^p(M(P))$ where $P$ is a non-$K$-$\mathfrak{F}$-subnormal subgroup of $G$ for $p\in\pi$. From Theorem \ref{Basic} it follows that this subgroup can be computed in a polynomial time.

  The algorithm for computing $G^{\mathrm{w}_\pi\mathfrak{F}}$ uses the same ideas.
\end{proof}

\subsection*{Acknowledgments}

I am grateful to A.\,F. Vasil'ev for helpful discussions.

{\small\bibliographystyle{siam}
\bibliography{BibAlg}}

\end{document}